\documentclass{amsart}
\usepackage{amsmath, amssymb}

\input xy
\xyoption{all}

\newcommand{\free}[1]{\underset{\scriptscriptstyle #1}{\displaystyle{\ast}}\,}

\title[Vertex exactness]{When universal edge-colored directed graph $C^*$-algebras are exact}

\author{Benton L. Duncan}

\address{Department of Mathematics\\
North Dakota State University\\
Fargo, North Dakota\\
USA}

\email{benton.duncan@ndsu.edu}



\begin{document}

\theoremstyle{plain}
\newtheorem{theorem}{Theorem}
\newtheorem{lemma}{Lemma}
\newtheorem{proposition}{Proposition}
\newtheorem{corollary}{Corollary}

\theoremstyle{definition}
\newtheorem{definition}{Definition}
\newtheorem*{algorithm}{Algorithm}
\newtheorem*{construction}{Construction}
\newtheorem*{example}{Example}

\theoremstyle{remark}
\newtheorem*{conjecture}{Conjecture}
\newtheorem*{acknowledgement}{Acknowledgements}
\newtheorem{remark}{Remark}

\begin{abstract}
We consider when the universal $C^*$-algebras associated to edge-colored directed graphs are exact.  Specifically, for countable edge-colored directed graphs we show that the universal $C^*$-algebra is exact if and only if the $C^*$-algebra is isomorphic to a graph $C^*$-algebra which occurs precisely when the universal and reduced $C^*$-algebras of the edge-colored directed graph are isomorphic. \end{abstract} 

\maketitle

\section{Introduction and preliminaries}

The idea of considering an edge-coloring function on a directed graph was considered in \cite{DuncanGraphProduct} and from a different perspective in \cite{Ara, AraGoodearlLeavitt, AraGoodearl}.  It is natural, given the study of directed graph algebras (for a survey of this study see \cite{Raeburn}), to consider the algebras associated to edge-colored directed graphs (called separated graphs in \cite{AraGoodearl}).  The natural focus then is on the $C^*$-algebras associated to an edge-colored directed graph.

In \cite{DuncanGraphProduct} the universal $C^*$-algebra for an edge-colored directed graph was studied using a representation of these algebras as universal free products of graph algebras.  This perspective allowed a natural extension of many results about graph algebras to the edge-colored directed graph context.  In this paper we return to this subject to investigate further some problems left open in \cite{DuncanGraphProduct}.  Specifically there are examples of edge-colored directed graphs which give rise to $C^*$-algebras which are not exact and one is left with the question of when this is the case.  

Since the $C^*$-algebras of edge-colored directed graphs are considered as universal free products, one approach to these questions is to consider exactness of free products.  An analysis of exactness for free products of finite dimensional algebras in \cite{DuncanExact} led us to reconsider the question in the context of edge-colored directed graph $C^*$-algebras.  In \cite{DuncanExact} one quickly realizes the role played by the amalgamating subalgebra in exactness.  This proved to be a useful point of view for edge-colored directed graph $C^*$-algebras.  

We now explain the main results of the paper. We first introduce an operation on an edge-colored directed graph, reversing edges, which produces a new graph with isomorphic associated $C^*$-algebras.  We then, proceeding in cases, consider a set of algorithms that allow us to either show that the associated $C^*$-algebra is not exact, or eventually turn the graph into a $1$-colored directed graph.  As the $1$-colored directed graph $C^*$-algebras are nuclear this completely answers the question concerning exactness and nuclearity for the universal $C^*$-algebra of an edge-colored directed graph. 

We will use the terminology from edge-colored directed graphs but our results can be translated to separated graph $C^*$-algebras.  It should be pointed out that in this paper we are focused solely on the universal $C^*$-algebra of an edge-colored directed graph.  These questions have also been investigated for the reduced algebra studied in \cite{AraGoodearl} where it is shown that the reduced $C^*$-algebra for a separated graph is always nuclear.  

Putting our results together with these we are also able to completely determine when the universal $C^*$-algebra of an edge-colored directed graph is isomorphic to the reduced $C^*$-algebra of the edge-colored directed graph.  Giving an answer to \cite[Problem 7.2]{AraGoodearl}, this occurs precisely when the universal $C^*$-algebra is nuclear.

\section{Edge-colored directed graph $C^*$-algebras}\label{background}

By an edge colored directed graph we mean a countable directed graph $G = (V,E, r, s)$ together with a coloring function $ \chi: E \rightarrow N$, where $N$ is any set.  In this context we will mean, by countable, that both the edge and vertex sets are countable.  Many of the results in this paper extend to non-countable graphs, however it is not clear that the algorithms we construct will apply in that context.  Since the edge set is countable we can assume that $\chi$ has range contained in $ \mathbb{N}$ and we will assume so throughout, and our notation will reflect this assumption.

An edge-colored Cuntz-Krieger family for an edge colored directed graph are collections of orthogonal projections $ \{ p_v: v \in V \}$ and partial isometries $\{ S_e: e \in E \}$ that satisfy the following properties:
\begin{enumerate}
\item $S_e^*S_e = P_{s(e)}$ for all $ e \in E$.
\item $ \sum_{\{ e: r(e) = v, \chi(e) = i\}} S_e S_e^* \leq p_v$, where this is an equality if $\{ e: r(e) = v, \chi(e) = i\}$ is finite otherwise it is a strict inequality.
\end{enumerate}

There is a universal $C^*$-algebra for edge-colored CK-families associated to an edge colored directed graph and we denote it by $C^*(G, \chi_G)$.  There is a standard construction of the algebra using free products as follows.  Denote by $G_i$ the subgraph given by $ (V, \{ e: \chi_G(e) = i \}, r, s )$ then $G_i$ is a directed graph and $C^*(G, \chi) = \free{i \in \mathbb{N}} C^*(G_i)$.  We will refer to the $G_i$ as the $1$-colored subgraphs of $G$.

We refer the reader to \cite{DuncanGraphProduct} for more details and other relevant results concerning the $C^*$-algebras of edge-colored directed graphs.  

In what follows $(G,\chi)$ will denote a fixed edge-colored directed graph with edge-coloring function $ \chi$.  Given a directed graph $G = (V,E, r, s)$ we can consider the underlying undirected graph.  This graph is not a graph in the traditional sense (i.e.\ ordered pairs indicating the presence of an edge between two vertices) since we will allow multiple edges between two vertices.

We will assume that the undirected graph is connected (i.e there is a path in the undirected graph between any two vertices in the graph).  Without this assumption we can just consider connected components of the undirected graph and see that the $C^*$-algebra of the edge-colored directed graph is the direct sum of the edge-colored directed graph for each component.  

\section{Some non-exact algebras}

We start with some (known) examples involving exactness of free products which will be helpful in analyzing the general situation.

\begin{proposition} \label{non-exact} The following universal free products are not exact:
\begin{enumerate} 
\item $C(\mathbb{T}) \free{\mathbb{C}} C( \mathbb{T})$.
\item $C(\mathbb{T}) \free{\mathbb{C}} \mathbb{C}^k$ with $k \geq 2$.
\item $ M^n \free{\mathbb{C}} M^m$, if $n, m \geq 2$.
\item $M^r \free {\mathbb{C}} M^s \free{ \mathbb{C}} M^t$, if $r, s, t \geq 2$.
\end{enumerate} 
\end{proposition}

\begin{proof}
We notice first that $C(\mathbb{T}) \cong C^*(\mathbb{Z})$ and that $\mathbb{C}^k = C^*(\mathbb{Z}_k)$ for all finite $k$.  

\begin{enumerate} 
\item We have that $C(\mathbb{T}) \free{\mathbb{C}} C( \mathbb{T}) \cong C^*(\mathbb{Z}) \free{\mathbb{C}} C^*(\mathbb{Z}) \cong C^*(\mathbb{Z} \free{} \mathbb{Z}) \cong C^*(\mathbb{F}_2)$ which is known to be not exact \cite{Wassermann}.
\item Next $ C(\mathbb{T}) \free{\mathbb{C}} \mathbb{C}^k \cong C^*(\mathbb{Z}) \free{\mathbb{C}} C^*(\mathbb{Z}_k) \cong C^*(\mathbb{Z} \free{} \mathbb{Z}_k)$.  Notice that if the generators of $\mathbb{Z}_k$ are $a_1, a_2, \cdots, a_k$ then consider $A = \mathbb{Z}$ and $ B = a_1 A a_1^{-1}$ inside $ \mathbb{Z} \free{} \mathbb{Z}^k$ and let $C$ be the subgroup generated by $A$ and $B$. Then $A$ and $B$ are both isomorphic to $ \mathbb{Z}$ and $C = A \free{} B subseteq \mathbb{Z} \free{\mathbb{Z}^k}$. Hence the associated $C^*$-algebra contains a subalgebra of the previous type and hence is not exact.
\item If $n = m = 2$ then this is \cite[Proposition 1]{DuncanExact}. If $m \geq 3$ and $n \geq 2$ then there are subalgebras $A \subseteq M_m$, $B \subseteq M_n$ both isomorphic to $M_2$ and then there is a quotient from this subalgebra generated by $A$ and $B$ onto $M_2 \free{\mathbb{C}} M_2$ which is the case $n = m = 2$.
\item This is \cite[Proposition 5]{DuncanExact}.
\end{enumerate}

\end{proof}

Note that in \cite[Theorem 2]{DuncanExact} it is shown that $M_2 \free{\mathbb{C}^2} M_2$ is exact as it is isomorphic to a graph $C^*$-algebra. This is the motivating example for the idea of reversing edges, that we introduce later.
 
Finally we have that $ \mathbb{C} \free{\mathbb{C}} \mathbb{C}^k \cong \mathbb{C}^k$ and hence any algebra of this form is also nuclear.

We will in what follows do a case by case analysis which will allow us to completely answer the question of exactness for $C^*(G, \chi)$.  For the most part this analysis rests on excluding subalgebras of the type in Proposition \ref{non-exact} using the projections associated to vertices in the directed graph.  We consider how such projections can ``propagate'' through the graph and then consider the added complications that cycles in the undirected graph give rise to.

\section{Reversing edges and subgraphs}

Let $(G,\chi)$ be a finite edge-colored directed graph.  If $H$ is a directed subgraph of $G$ then $ (H, \chi)$ the edge-colored directed graph with coloring given by restriction.  We will call $(H,\chi)$ an edge-colored subgraph of $(G, \chi)$. Notice that $(H, \chi)$ is an edge-colored directed graph in its own right and we will denote this graph with the same symbol.  We have the following results connecting the $C^*$-algebra of an edge-colored subgraph to a $C^*$-subalgebra of the original graph $C^*$-algebra.

\begin{theorem}\label{subgraph} Let $(H,\chi) \subseteq (G,\chi)$ and let $A$ be the subalgebra of $C^*(G,\chi)$ generated by $ \{ S_e: e \in E(H) \}$ and $ \{ p_v: v \in V(H) \}$, then there is a surjection $\pi: A \rightarrow C^*(H,\chi)$. \end{theorem}

\begin{proof}
We first prove this for $1$-colored graphs and then use the free product decomposition of an $n$-colored graph to extend to the general case.

So assume that $H$ is a subgraph of $G$ and $G$ is $1$-colorable via the function $\chi$.  Now notice that the subalgebra $A$ will be generated by partial isometries satisfying the following relations:
\begin{enumerate}
\item The $p_v$ are mutually orthogonal nonzero projections.
\item $S_e^*S_e= P_{s(e)}$.
\item $\sum_{ r(e) = v } S_eS_e^* \leq P_{v}$.
\end{enumerate}

In addition the gauge action on $C^*(G)$ will reduce to a gauge action on the subalgebra $A$, and hence the family $ \{ p_v, S_e \}$ generating $A$ is a gauge-invariant Toeplitz-Cuntz-Krieger family.  Since the graph algebra $C^*(H)$ is co-universal for gauge-invariant Toeplitz-Cuntz-Kreiger families (see \cite{Katsura}, \cite{Sims}, and \cite{SimsWebster} for a general discussion of co-universality of graph algebras) it follows that there is a $*$-representation $\pi: A \rightarrow C^*(H)$ which is onto.

Now if $(G,\chi)$ is $n$-colored then there exist $1$-colored graphs $\{ G_i \}_{i=1}^n$ such that $C^*(G, \chi) = \free{P} C^*(G_i)$.  Consider $A_i = A \cap C^*(G_i)$ which is generated by $ P' = \{ p_v: v \in V(H) \}$ and $ \{ S_e: e \in E(H) \cap E(G_i) \}$.  Then $A = \free{P'}A_i$ by \cite[Theorem 3]{DuncanGraphProduct} (see also \cite[Theorem 4.2]{Pedersen}).  Now for each $i$ there is $\pi_i: A_i \rightarrow C^*(H_i)$, where $H_i$ is $H \cap G_i$ which when restricted to the subalgebra $P'$ all coincide.  It follows that $ \free{} \pi_i : \free{P'} A_i \rightarrow \free{P'} C^*(H_i)$. The former algebra is of course $A$ and the latter is $C^*(H,\chi)$.
\end{proof}

\begin{definition} We say that $(H,\chi)$ is a full edge-colored subgraph of $(G,\chi)$ if the map in the previous theorem is an injection. \end{definition}

For a $1$-colored graph the full subgraphs of $G$ are given by those subgraphs such that $ \{ e \in E(H): r(e) = v \} = \{ e \in E(G): r(e) = v \}$ for every vertex in $H$ (This is just an application of the the gauge-invariant uniqueness theorem for arbitrary graphs, see \cite[Theorem 2.1]{BatesHongRaeburnSymanski}).  A similar result is true of an $n$-colored graph.

\begin{proposition} The edge-colored graph $(H,\chi) $ which is a subset of the edge-colored graph $ (G,\chi)$ is full if and only if when $ \{ e \in H: r(e) = v \} $ is nonempty then $\{ e \in G: \chi(e) = i, r(e) = v \} = \{ e \in H: \chi(e) = i, r(e) = v \}$. \end{proposition}

\begin{proof}
This is just the fact that the third defining relation for a Toeplitz-Cuntz-Krieger family gives a Cuntz-Krieger family if and only if the CK-inequality is in fact equality for any vertex receiving finitely many edges.  
\end{proof}

In this section we will assume that $G = (V,E,r,s)$ with both $E$ and $V$ being countable. 

We start with the fact that $C^*(G, \chi)$ is generated by the sets $ \mathcal{P} := \{ p_v, v \in V \}$ and $ \mathcal{S} := \{ S_e: e \in E \}$, and $ \mathcal{S}^*:= \{ S_e^*: e \in E \}$ with the following relations: 
\begin{itemize} 
\item $p_v S_e = \begin{cases} S_e & \mbox{ if } r(e) = v \\ 0 & \mbox{ otherwise} \end{cases}$.
\item $S_e p_v = \begin{cases} S_e & \mbox{ if } s(e) = v \\ 0 & \mbox{ otherwise} \end{cases}$.
\item $S_eS_f = 0 \mbox{ if } r(f) \neq s(e)$ otherwise it is nonzero.
\item $S_e^*S_f = 0 \mbox{ if } r(e) \neq r(f) \mbox{ or } r(e) = r(f), e \neq f, \mbox{ and } \chi(e) = \chi(f)$ otherwise it is nonzero.  Also if $e = f$ then this is $P_{s(e)}$.
\item $S_eS_f^* = 0 \mbox{ if } s(e) \neq s(f)$ otherwise it is nonzero.  If $e =f$ then this is a subprojection of $P_{r(e)}$.
\end{itemize}

That the ``nonzero'' products above are non-zero one can consider the fact that the products are non-zero in the Leavitt path algebra associated to $ (G, \chi)$ \cite{AraGoodearlLeavitt} and that the Leavitt path algebra injects into $C^*(G< \chi)$ applying \cite[Theorem 3.8]{AraGoodearl}. However, those products of $S_e$ and $S_f^*$ which are nonzero in the preceding list need not give rise to a partial isometry, since we don't know a priori that the range projections of the $S_e$ commute (in fact they often do not, see \cite{AraExel} where they consider a variant of the separated graph algebras where they quotient by the commutators of non-commuting range projections).  An important fact we will use, however, is that when we know that the family $ \{ S_eS_e^*: e\in E \}$ consists of mutually commuting projections any such product will, in fact, be a partial isometry \cite{AraExel}

For a vertex $v$ we define the vertex degree of the vertex to be equal to an $n$-tuple $d_v = (a_1, a_2, \cdots, a_n )$ where $a_i$ is the number of edges labeled by $i$ which end at the vertex $v$. We will allow $a_i = \infty$ in the case that the set of edges colored $i$ is countably infinite.  We also allow $d_v$ to be a sequence (in the case that $v$ receives edges of infinitely many different colors).

\begin{remark} \label{ordering} Notice that by \cite[Theorem 2]{DuncanGraphProduct} we can without loss of generality assume that $a_1 \geq a_2 \cdots \geq a_n$. We will do so implicitly throughout (with two exceptions that occur inside proofs to simplify the arguments and will be noted explicitly).\end{remark}  

We now introduce a construction which will allow us to ``reverse'' certain edges in the graph without affecting the associated $C^*$-algebra.  

Let $(G,\chi) $ be an edge-colored directed graph with $e \in E(G)$.  Construct a new graph by reversing the edge $e$, call it $G_{e}$.  Formally we have $V(G) = V(G_{e})$, $E(G_e) = (E(G) \setminus \{ e \}) \cup \{\overline{e} \}$, $r_{G_e}(f) = r(f)$ and $s_{G_e}(f) = s(f)$ for all $ f \in E(G) \setminus \{ e \}$, and $r_{G_e}(\overline{e}) = s(e)$ and $s_{G_e}(\overline{e}) = r(e)$.  Next define $\chi_{G_e}(f) = \chi(f)+1$ for all $f \in E(G) \setminus \{ e \}$ and $\chi_{G_e}(\overline{e}) = 1$.  We say that $ (G_e, \chi_{G_e})$ is the graph obtained from $(G,\chi)$ by reversing the edge $e$.  Note that the new graph involves changing the color of the edge $e$.  In particular, reversing an edge $e$ and then reversing the reversed edge may not yield the original graph.

\begin{example}
Consider the one-colored graph 
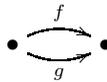
\[ \xymatrix{ {\bullet} \ar@/^/[r]^{f} \ar@/_/[r]_{g} & {\bullet} } \]
Reversing the edge $g$ gives a graph which is equivalent to 
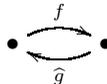
\[ \xymatrix{ {\bullet} \ar@/^/[r]^{f} & {\bullet} \ar@/^/[l]^{\widehat{g}} }\]
and then reversing the edge $\widehat{g}$ yields the graph 
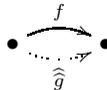
\[ \xymatrix{ {\bullet} \ar@/^/[r]^{f} \ar@{.>}@/_/[r]_{\widehat{\widehat{g}}} & {\bullet} } \]
In the first case the associated $C^*$-algebra is $M_3$ and the third graph gives rise to the $C^*$-algebra $M^2 \free{\mathbb{C}^2} M^2$ which is not finite dimensional.
\end{example}

The next proposition now yields information about the new $C^*$-algebra when an edge is reversed.

\begin{proposition}\label{reversing} Let $(G,\chi)$ be an edge-colored directed graph and $e$ an edge in $G$.  If $\chi_G(e) \neq \chi_G(g)$ for any edge $g$ with $r(g) = r(e)$ then $C^*(G,\chi)$ is isomorphic to $C^*(G_e, \chi_{G_{e}})$. \end{proposition}

\begin{proof}
Without loss of generality \cite[Theorem 2]{DuncanGraphProduct} we 

will assume that $ \chi(e) = 1 $, $\chi_{G_e}(\overline{e}) = 1 $ and $ \chi(g) \neq 1$ for any $g \neq e$ in $E(G)$.  Similarly we assume $
\chi_{G_e}(\overline{e}) = 1 $ and $ \chi_{G_e}(g) \neq 1$ for any $g 
\neq \overline{e}$ in $E(G_e)$.  Then by \cite[Theorem 1]{DuncanGraphProduct} we have $C^*(G, \chi) = C^*(G_1) \free{P} \left( \free{P}  
\{ C^*(G_i): i \neq 1 \} \right)$ and $C^*(G_e, \chi_{G_e}) = 
C^*((G_e)_1) \free{P} \left( \free{P} \{ C^*((G_e)_i): i \neq 1 \} 
\right)$.  By construction we have that $(G_e)_i = G_i$ for all $ i \neq 1$.  And if we let $V' = V(G) \setminus \{ s(e), r(e) \}$ then $G_1 = M_2 \oplus \mathbb{C}^{|V'|}$ and $(G_e)_1 = M^2 \oplus \mathbb{C} ^{|V'|}$.  It follows that \begin{align*} C^*(G, \chi) & = \left( M_2 \oplus \mathbb{C}^{|V'|} \right) \free{P} \left( \{ C^*(G_i): i \neq 1 \} \right) \\ & = \left( M_2 \oplus \mathbb{C}^{|V'|} \right) \free{P} \left( \{ C^*((G_e)_i): i \neq 1 \} \right) \\ & = C^*(G_e, \chi_{G_e}). \end{align*}
\end{proof}

The same proof applies to the reduced edge-colored directed graph $C^*$-algebras using the reduced free product rather than the universal free product (as constructed in \cite[Definition 3.5]{AraGoodearl}).

\begin{proposition}\label{reduced} Let $(G,\chi)$ be an edge-colored directed graph and $e$ an edge in $G$.  If $\chi_G(e) \neq \chi_G(g)$ for any edge $g$ with $r(g) = r(e)$ then $C_r^*(G,\chi)$ is isomorphic to $C_r^*(G_e, \chi_{G_{e}})$. \end{proposition}

In the case that $r(e) = s(e)$ this theorem just gives rise to a trivial re-coloring of the edges and hence doesn't provide any useful change in the graph, however if $r(e) \neq s(e)$ then this operation can be used to simplify some graphs. 

We will say that an edge $e$ is reversible if $ \{ f: \chi(e) = \chi(f) \mbox{ and } r(e) = r(f ) \} = \{ e \}$. In other words, these are the edges for which the previous theorem applies. We denote by $E_{\rm rev}$ the set of all reversible edges in the graph.  Similarly we say that a vertex $v$ supports an irreversible edge if there is an edge $ e \not\in E_{\rm rev}$ with $r(e) = v$.  We denote by $V_{\rm irr}$ the set of vertices which support an irreversible edge.

\section{When is $C^*(G, \chi)$ exact/nuclear}

We are now in a position to consider exactness/nuclearity of the $C^*$-algebra of an edge-colored directed graph.  We proceed in cases depending on the set $V_{\rm irr}$. In each case we will construct an algorithm that either ends when the algebra is not exact, or exhausts the possible finite graphs. 

\subsection{The set $V_{\rm irr}$ is empty}

Let $(G, \chi)$ be a graph in which every edge is reversible (this is equivalent to $V_{\rm irr}$ is empty).  In this case $S_eS_e^* = P_{r(e)}$ for every edge $e$ and hence any product of the generating partial isometries will be a partial isometry.  We construct an alorithm in which the vertex set is unchanged throughout, but in each iteration of the algorithm we will replace $(G, \chi)$ with a directed graph (with isomorphic $C^*$-algebra) with some edges reversed.

\begin{algorithm} \label{reversible} {\em Algorithm for Reversible Edges}

Base step: 

Fix a vertex $v \in V$.  Let $V_0 = \{ v \}$, $E_0^{\rm loop} = \{ e \in G: r(e) = s(e) = v \}$ and let $E_0 = \{ e \in E \setminus E_0^{\rm loop}: r(e) \in V_0 \} $. Also let $m_1 = | E_0^{\rm loop}|$.

Step 1: 

Let $(G_1, \chi)$ be the edge-colored directed graph formed by reversing each of the edges in $E_0$. Let $V_1 = \{ w: r(e)= w, s(e) = v \mbox{ for some } e \in E(G_1) \} \setminus V_0$. Now $E_1^{\rm loop} = \{ (e,f) \in E(G_1): s(e), s(f) \in V_0, r(e) = s(e) \in V_1 \}$ and let $E_{1'}^{loop}$ be the set of all edges $e \in E_{1}$ such that $s(e) \in V_{1}, r(e) \in V_{1}$ but $ (e,f) \not\in E_{1}^{loop}$ for any $f \in E_{i-1}$. Let $E_1 = \{ e \in E(G_1): r(e) \in V_1 \} \setminus \left( \{ e: (e,f) \in E_1^{\rm loop} \mbox{ for some } f \} \right) \cup E_{1'}^{\rm loop}$. Finally we let $m_2 = | E_1^{\rm loop} | + |E_{1'}^{\rm loop}|$.

Step $i$: 
 
Let $(G_i, \chi)$ be the edge-colored directed graph formed by reversing each of the edges in $E_{i-1}$. Let $V_i \{ w: r(e) = w, s(e) = v \mbox{ for some } e \in E(G_{i-1}) \} \setminus V_{i-1}$. Let $E_i^{\rm loop} \{ (e,f) \in E(G_{i}): s(e), s(f) \in V_{i-1}, r(e) = s(e) \in V_i \}$ and $E_{i'}^{loop}$ be the set of all edges $e \in E_{i-1}$ such that $s(e) \in V_{i-1}, r(e) \in V_{i-1}$ but $ (e,f) \not\in E_{i}^{loop}$ for any $f \in E_{i-1}$. Let $E_i = \{ e \in E(G_i): r(e) \in V_i \} \setminus \left(\{ e: (e,f) \in E_i^{\rm loop} \mbox{ for some } f \}\right) \cup E_{i'}^{\rm loop}$. Now let $m_i = | E_i^{\rm loop}| + |E_{i'}^{\rm loop} |$.

As the original underlying directed graph is connected, and the graph is finite then this algorithm will terminate.  As a final step we let $M_{ (G, \chi) } = \sum m_i$. 
\end{algorithm}

We have the following theorem which completely describes exactness/nuclearity in this context.

\begin{theorem} \label{reversingtheorem} Let $(G, \chi)$ be an edge-colored directed graph in which every edge is reversible. Then $C^*(G, \chi)$ is nuclear if $m_{(G,\ chi)} \leq 1 $ and $C^*(G,\chi)$ is not exact if $M_{(G,\chi)} \geq 2$. \end{theorem}

\begin{proof}
We begin with the case in which $m_{(G, \chi)} = 0$.  In this case at the end of the algorithm we are left with a graph in which no two edges share a common range. It follows that there is a recoloring of the graph as a $1$-colored graph and hence the $C^*$-algebra is isomorphic to the $C^*$-algebra of a directed graph.  Such $C^*$-algebras are always nuclear \cite[Remark 4.3]{Raeburn}.

If $M_{(G, \chi)} = 1$ then either there is an edge in $E_{i'}^{\rm loop}$ or a pair of edges $(e,f) \in E_i^{\rm loop}$.  In either case, after the algorithm is complete we are left with the situation that either $e$ is a loop or there are two edges $e, f$ with $r(e) = r(f)$. 

If $e$ is itself a loop then redo the algorithm choosing the inital vertex to be the range of $e$.  Since there are no other ``loop'' edges we will end up with a graph with no two edges sharing a common range, and hence a $1$-colored graph. Then the $C^*$-algebra will be isomorphic to a nuclear graph $C^*$-algebra.

If instead we have that there are two edges $e$ and $f$ with $r(e) = r(f) = w$ then no other edges share a common range. There is then two directed paths from the fixed vertex $v$ to $w$ one of which ends with the edge $e$ and the other ends with the edge $f$. If one then reverses the edges in one of these two paths then the new graph will be a $1$-colored graph and again we are left with a $C^*$-algebra isomorphic to a nuclear graph $C^*$-algebra.

On the other hand, if $M_{(G, \chi)} \geq 2$ then we have at least one of the following graph possibilities: \begin{enumerate}
\item at least two edges which are loops based at the vertex $v$; \item one edge which is a loop based at the vertex $v$ and there is a pair of edges $e$ and $f$ with $r(e) = r(f)$; 
\item there are three edges $e, f,$ and $g$ with $r(e) = r(f) = r(g)$; \item or there are four edges $e,f, g,$ and $h$ with $r(e) = r(f)$ and $r(g) = r(h)$. \end{enumerate} 

In each of these cases we will identify a subalgebra which maps onto a copy of $C^*(\mathbb{F}_2)$.  To do this let $U$ and $V$ denote the unitary generators of $C^*(\mathbb{F}_2)$.  If $n = |V(G)|$ we will first identify an edge-colored Cuntz-Krieger family in $\mathcal{A} = M_n(C^*(\mathbb{F}_2))$.  The identification depends on which of the cases we have above.  For notation sake we will write $E_{i,j})$ to be the matrix in $\mathcal{A}$ with a $1_{C^*(\mathbb{F}_2))}$ in the $i-j$ entry and zeroes everywhere else.  Similarly we will write $U_{i,j}$ and  $V_{i,j}$ for the matrix with $U$ ($V$ respectively) in the $i,j$ entry and zeroes everywhere else.

For the first graph possibility let $e$ be the first loop based at $v$ and $g$ the second loop based at $v$.  We will write $A = S_e$ and $B = S_g$ inside $C^*(G, \chi)$.

In the second situation there is a loop edge $g$ based at $v$ and two directed paths $\mu = e_1e_2\cdots e_ne$ and $ \nu= f_1f_2\cdots f_mf$. Then let $A = S_{e_1}S_{e_2} \cdots S_{e_n} S_e S_f^*S_{f_{m}}^* \cdots S_{f_1}^*$ and $B = S_g$ in the $C^*$-algebra $C^*(G, \chi)$.

In the third situation we have three paths $ \mu = e_1e_2 \cdots e_ne$, $\nu = f_1f_2, \cdots f_mf$ and $\tau = g_1g_2\cdots g_kg$ all of which begin at $v$ and end at the common vertex $w = r(e) = r(f) = r(g)$.  And consider the two elements of $C^*(G, \chi)$ given by $A = S_{e_1}S_{e_2} \cdots S_{e_n} S_e S_f^*S_{f_{m}}^* \cdots S_{f_1}^*$ and $B = S_{g_1}S_{g_2} \cdots S_{g_k} S_g S_f^*S_{f_{m}}^* \cdots S_{f_1}^*$.

In the fourth situation we have four paths $ \mu = e_1e_2 \cdots e_ne$, $\nu = f_1f_2, \cdots f_mf, \tau = g_1g_2\cdots g_kg,$ and $\sigma = h_1 h_2 \cdots h_lh$ all of which begin at $v$ and two of which end at the common vertex $w = r(e) = r(f)$ and two of which end at the common vertex $u = r(g) = r(h)$.  Here we let $A = S_{e_1}S_{e_2} \cdots S_{e_n} S_e S_f^*S_{f_{m}}^* \cdots S_{f_1}^*$ and $B = S_{g_1}S_{g_2} \cdots S_{g_k} S_g S_h^*S_{h_{l}}^* \cdots S_{h_1}^*$.

Now to each vertex in $G$ we assign a unique element $E_{i,i}$ in $M_n(C^*(\mathbb{F}_2))$.  If $x$ is the vertex we will write $E_{x,x}$ for the assigned element.  Next for any edge $d$ not equal to $e$ or $g$ we assign the partial isometry $d \mapsto E_{s(d),r(d)}$.  Finally we map $e$ to $U_{s(e),r(e)}$ and $f$ to $V_{s(f),r(f)}$.  It is straightforward to see that these assignments give rise to an edge-colored CK-family associated to the graph and that hence there is a $*$-homomorphism $\pi: C^*(G, \chi) \rightarrow M_n(C^*(\mathbb{F}_2))$.  Notice that under this homomorphism $A$ is sent to $U_{v,v}$ and $B$ is sent to $V_{v,v}$.  It follows that the range of $\pi$ contains a copy of $C^*(\mathbb{F}_2)$ and hence $C^*(G, \chi)$ is not exact. 

\end{proof}

\begin{remark} The value of $M_{(G, \chi)}$ is equal to the topological genus of the underlying undirected graph. For more on these graphs, including a complete description of the associated $C^*$-algebras we refer the reader to \cite{BrownleeDuncan} \end{remark}

Restating this result we have the following:

\begin{corollary} Let $(G, \chi)$ consist of only reversible edges. Then $C^*(G, \chi)$ is nuclear if and only if there a $1$-colorable graph which is obtained by reversing edges in $(G, \chi)$. Otherwise $C^*(G, \chi)$ is not exact. \end{corollary}

\subsection{The set $V_{\rm irr}$ is a singleton}

Assume that $ v \in V_{\rm irr}$ and let $d_v = (a_1, a_2, \cdots, a_k)$ be the vertex degree of $v$, and as in Remark \ref{ordering} we have $a_1 \geq a_2 \geq \cdots \geq a_k$.  Since exactness is preserved by subalgebras we will often fix a vertex $v$ and focus on the subalgebra $p_vC^*(G, \chi)p_v$ which if we can show this is non-exact will tell us about the overall algebra. This allows us to rule out a lot of graphs.

\begin{proposition}\label{vertex-degree} If $C^*(G, \chi)$ is exact and $v \in V_{\rm irr}$ has vertex degree \[d_v = (a_1, a_2, \cdots, a_k)\] then we have that $a_2 \leq 1$.
\end{proposition}

\begin{proof}
Notice that $p_vC^*(G, \chi)p_v$ contains $p_vC^*(G_1) \free{\mathbb{C}} C^*(G_2)\free{\mathbb{C}} \cdots \free{\mathbb{C}} p_vC^*(G_n)p_v$ \cite{Armstrongetal}, where $\mathbb{C}$ in the free products is the subalgebra generated by $p_v$.  We proceed in cases:

If $s(e) \neq r(e)$ for any $e$ which shares a color with another edge  then $p_vS^*(G_i) P_v$ is isomorphic to $M_{a_i}$.  It follows that $C^*(G, \chi)$ contains a subalgebra of the form $M_{a_1} \free{\mathbb{C}} M_{a_2} \free{\mathbb{C}} \cdots \free{\mathbb{C}} M_{a_n}$ which , since $a_1 \geq a_2 \geq 2$ is not exact by Proposition \ref{non-exact}.

Consider edges $e$ with $s(e) = v = r(e)$ and notice that $p_vC^*(G_i) p_v$ contains an algebra which has quotient equal to the Toeplitz algebra (consider the algebra generated by $S_e$.  The subalgebra $P_vC^*(G_i)p_v$ then contains an algebra which has a quotient isomorphic to $ C( \mathbb{T})$.  Then $p_vC^*(G, \chi)p_v$ contains a copy of either $M_2 \free{\mathbb{C}} A$ or $A \free{\mathbb{C}} A$ where $A$ surjects onto $C(|mathbb{T})$.  Since exactness is preserved by surjections combined with Proposition \ref{non-exact} it follows that $p_vC^*(G, \chi)p_v$ is not exact.
\end{proof}

We will say that a vertex whose vertex degree is of the form $(a_1, a_2, \cdots, a_n)$ with $a_1 \geq a_2 \cdots a_n$ and $a_2 = 1$ is {\em vertex-exact}.  Then another way of phrasing these propositions is that if the vertex is not vertex-exact then $C^*(G, \chi)$ is not exact.  We will see in what follows that although this is a necessary condition it is not sufficient. Keep in mind though, that if the graph is vertex-exact then the set of partial isometries $\{ S_e S_e^*: e \in E(G) \}$ is commutative and hence we know (in this case) that any finite product of the generating partial isometries is a partial isometry.

Let $V_{\rm irr} = \{ v_1, v_2, \cdots, v_k \}$ and assume that $d_{v_i} = (a_{i,1}, a_{i,2}, \cdots, a_{i,m})$ with $a_1 \geq 2$ and $a_2 \leq 1$ (i.e.\ each of the vertices is vertex exact but none of the vertices are reversible.). We let $E_{v_i}$ be the set of edges with range $v_{i}$ which $ \chi$ maps to $1$. We let $E_v = \cup E_{v_i}$.

\begin{algorithm} \label{singleton} {\bf Algorithm Based at Irreducible Vertices}

Base step: 

Set $V_0 = \{ v_1 \}$.

Base sub-step: Set $V' = \{ w: w = s(e), r(e) \in V_0, e \not\in E_v \}$.  Let $E_{0}^{\rm loop} = \{ e : s(e) = r(e) \in V_0 e \not \in E_v \}$. Let $E_0 = \{ e: s(e) \in V_0, e \not\in E_v \cup E_{0}^{\rm loop} \}$.

1st sub-step: Let $(G_1, \chi)$ be the edge-colored directed graph formed by reversing all of the edges in $E_0$. Let $V_1 = \{ v: r(e) = v, s(e) \in E_0 \}$. Let $E_{1}^{\rm loop} = \{ (e, f): e,f \in E(G_1) \setminus E_v, r(e) = r(f), s(e), s(f) \in V_0 \}$ and $E_{1'}^{\rm loop} = \{ e: r(e), s(e) \in V_1 \}$. Then let $E_1 = \{ e:  s(e) \in V_1, e \not\in E_v \cup E_1^{\rm loop} \cup E_{1'}^{\rm loop} \}$.

$i$th sub-step: Let $(G_i, \chi)$ be the edge-colored directed graph formed by reversing all of the edges in $E_{i-1}$. Let $V_i = \{ v: r(e) = v, s(e) \in E_{i-1} \}$. Let $E_{i}^{\rm loop} = \{ (e, f): e,f \in E(G_i) \setminus E_v, r(e) = r(f), s(e), s(f) \in V_{i-1} \}$ and $E_{i'}^{\rm loop} = \{ e: r(e), s(e) \in V_i \}$. Then let $E_i = \{ e:  s(e) \in V_i, e \not\in E_v \cup E_i^{\rm loop} \cup E_{i'}^{\rm loop} \}$.

Continue repeating sub-steps until you get to a point where $E_i = \emptyset$.

$i$-th step: 

Set $V_0 = \{v_k \}$ where $k$ is the smallest value such that $v_k \not\in V_i$ for any $0 \leq i $ in any of the previous sub-steps. Then repeat the sub-steps for the new $V_0$ until the $E_i$ sets are empty.

Continue repeating steps until every vertex $v_i \in V_{\rm irr}$ has appeared in one of the $V_i$ in at least one sub-step for a step.

As $ V_{\rm irr}$ is finite this process will eventually terminate.
We call the set of loops identified by the first algorithm the set of {\em reversible loops}. 

\end{algorithm}

It is possible after completing this algorithm that there are edges not in $E_v$ that have not appeared in any of the $E_i$ sets.  However, since the graph is connected then the set of edges which are not in $E_v$ and have not appeared in any $E_i$ must connect to one of the vertices in $V_{\rm irr}$.  This connection must occur through (at least) one of the elements of $E_v$, and not through a sequence of reversible edges that ends in $V_{\rm irr}$.  However, by definition all such edges are reversible.  We call the set of such edges the set of edges avoiding $V_{\rm irr}$.  We proceed with a second algorithm for these edges; it is essentially the same algorithm for the case where $V_{\rm irr}$ is empty with some slight modifications.

\begin{algorithm} \label{part2} {\bf Algorithm for Elements of $E_v$}

Base step: 

Fix a vertex $v$ which is not in $V_i$ for any $i$ and such that $ v = s(e)$ for some $e \in E_v$.  Let $V_0 = \{ v \}$, $E_0^{\rm loop} = \{ e \in G: r(e) = s(e) = v \}$ and let $E_0 = \{ e \in E \setminus E_0^{\rm loop}: r(e) \in V_0 \} $. Also let $m_1 = | E_0^{\rm loop}|$.

Step 1: 

Let $(G_1, \chi)$ be the edge-colored directed graph formed by reversing each of the edges in $E_0$. Now let $V_1 = \{ w: r(e)= w, s(e) = v \mbox{ for some } e \in E(G_1) \setminus E_r\} \setminus V_0$. Now $E_1^{\rm loop} = \{ (e,f) \in E(G_1): s(e), s(f) \in V_0, r(e) = s(e) \in V_1 \}$ and let $E_{1'}^{loop}$ be the set of all edges $e \in E_{1}$ such that $s(e) \in V_{1}, r(e) \in V_{1}$ but $ (e,f) \not\in E_{1}^{loop}$ for any $f \in E_{i-1}$. Let $E_1 = \{ e \in E(G_1): r(e) \in V_1 \} \setminus \left( \{ e: (e,f) \in E_1^{\rm loop} \mbox{ for some } f \} \right) \cup E_{1'}^{\rm loop}$. Finally we let $m_2 = | E_1^{\rm loop} | + |E_{1'}^{\rm loop}|$.

Step $i$: 
 
Let $(G_i, \chi)$ be the edge-colored directed graph formed by reversing each of the edges in $E_{i-1}$. Let $V_i \{ w: r(e) = w, s(e) = v \mbox{ for some } e \in E(G_{i-1} \setminus E_r \} \setminus V_{i-1}$. Let $E_i^{\rm loop} \{ (e,f) \in E(G_{i}): s(e), s(f) \in V_{i-1}, r(e) = s(e) \in V_i \}$ and $E_{i'}^{loop}$ be the set of all edges $e \in E_{i-1}$ such that $s(e) \in V_{i-1}, r(e) \in V_{i-1}$ but $ (e,f) \not\in E_{i}^{loop}$ for any $f \in E_{i-1}$. Let $E_i = \{ e \in E(G_i): r(e) \in V_i \} \setminus \left(\{ e: (e,f) \in E_i^{\rm loop} \mbox{ for some } f \}\right) \cup E_{i'}^{\rm loop}$. Now let $m_i = | E_i^{\rm loop}| + |E_{i'}^{\rm loop} |$.

For the vertex $v$ we define the sum $M_v = \sum m_i$. We then repeat  the algorithm for every vertex $w$ which is a source of an edge in $E_r$ that has not already appeared in one of the vertex sets in either algorithm, each time finding a value $M_w$.  As the original underlying directed graph is connected, and the graph is finite then these algorithms will eventually terminate and include every edge and vertex in the original graph. 
\end{algorithm}

We have the following proposition which completely describes exactness and nuclearity in the context where $V_{\rm irr}$ is a singleton.

\begin{theorem} \label{singletonexact} Let $(G, \chi)$ be an edge-colored directed graph in which $V_{\rm irr} = \{ v \}$ and assume that $v$ is vertex-exact.  Then $C^*(G, \chi)$ is not exact if 
\begin{itemize}
\item $(G, \chi)$ contains a reversible loop.
\item There is an edge $e$ in $E_v$ such that $M_{s(e)} \geq 2$.
\end{itemize}
otherwise $C^*(G, \chi)$ is nuclear.
\end{theorem}

\begin{proof}
Since $ v \in V_{\rm irr}$ we know that there are at least two edges with range $v$ which share a color, call these edges $e$ and $f$.  It follows that $p_vC^*(G_1)p_v$ contains a copy of $ \mathbb{C}^k$ where $k$ is the number of edges with color $1$ and range $v$.

If $(G, \chi)$ contains a reversible loop then there is a nontrivial path in the underlying undirected graph which begins and ends at $v$ and does not include any edge in $E_v$.  Let $Z$ be the partial isometry in $C^*(G, \chi)$ which traces out this path and assume, without loss of generality that $\chi$ colors every edge in this path with a $2$, and $ \chi$ colors any edge not in this path with a value other than $2$. Then $G_2$ consists of a singly colored graph with one nontrivial cycle.  Now $p_vC^*(G_2)p_v$, as in the proof of the second case of \ref{reversingtheorem} contains a copy of $C(\mathbb{T})$ generated by $X$.  We now have that $p_vC^*(G, \chi)p_v$ contains a copy of $p_vC^*(G_1)p_v \free{\mathbb{C}} p_v C^*(G_2)p_v$ \cite{Armstrongetal} which contains a copy of $\mathbb{C}^k \free{\mathbb{C}} C(\mathbb{T})$ which is not exact, and hence $C^*(G, \chi)$ is not exact.

Now if $(G, \chi)$ does not contain a reversible loop then after completing Algorithm \ref{singleton} we are left with a graph which is $1$-colorable except (potentially) for the subgraphs which connect to $v$ only through $E_v$ (i.e.\ the subgraphs that are dealt with in Algorithm \ref{part2}.  However for each of these subgraphs we use the same arguments as in Theorem \ref{reversingtheorem} to see that if there is $e$ with $S_{s(e)} \geq 2$ then the algebra is not exact and otherwise we have a collection of reversing of edges which turns $(G, \chi)$ into a $1$-colorable graph and hence the $C^*$-algebra is nuclear.
\end{proof}

A variation of the previous then applies to graphs where $V_{\rm irr}$ is not a singleton.  The extra complication is presented by the fact that two elements of $V_{\rm irr}$ may be connected via a path.  Assume that there are edges $e, f \in E_v$.  If after completing the algorithms there is a directed path $ \mu = v_1v_2\cdots v_n$ such that $s(\mu) = s(e)$ and $r(\mu) = s(f)$, then we say that there is a {\em path connecting $e$ and $f$}. 

\subsection{The set $V_{\rm irr}$ contains two or more elements}

As in the previous subsection if any element of $V_{\rm irr}$ is not vertex-exact then $C^*(G, \chi)$ is not exact, so for our purposes we will assume throughout this subsection that every element of $V_{\rm irr}$ is vertex-exact and hence any finite sequence of generating partial isometries is a partial isometry itself.

We will assume the algorithms as described have already been applied to the graph $(G, \chi)$.  Here we have to be a little more careful about paths connecting pairs of edge $e$ and $f$ both elements of $E_v$ since such paths may ``connect'' two distinct elements of $V_{\rm irr}$. We also have to deal with directed paths that connect two vertices in $V_{\rm irr}$.  We will treat these cases as one. Let $ \mu = e_1e_2\cdots, e_n$ be a path in the undirected graph which connects two vertices $v, w \in V_{\rm irr}$. Notice that, without loss of generality $ r(e_1) = v$ or $s(e_1) = v $ and $r(e_n) = w$ or $s(e_n) = w$.  (it is possible in this construction that $ v = w$).  We will say that $\mu$ is reduced if $s(e_1) \neq r(e_1), s(e_n) \neq r(e_n)$, and $e_2, e_3, \cdots, e_{n-1}$ are reversible.  Essentially, we don't want $v_1$ or $v_n$ to be loops and we don't want $\mu$ to pass through any elements of $V_{\rm irr}$ except at the start and end of the path.

Now we say that $ \mu$ has a {\em reversible end} if $ v_1$ or $v_n$ is reversible.  Without loss of generality we will assume $v_1$ is reversible and in this case we can reverse any of the edges in $\mu$ so that there is a directed path from $v$ to $w$. On the other hand, if neither $e_1 $ nor $e_n$ are reversible then we can form a directed path in the graph by reversing the other edges (as needed) so that $v_2v_3\cdots v_n$ is a directed path from $s(v_1) $ to $w$, or alternatively a directed path from $s(v_n)$ to $v$.

Assume that $(G, \chi)$ is vertex-exact and that that $\mu$ has a reversible end which gives rise to a directed path from $v$ to $w$ (i.e. $e_1$ is reversible) with $\chi(e_n) = k$. If $v$ and $w$ have vertex degree $(a_1, a_2, \cdots, a_n)$ and $(b_1, \cdots, b_m)$, respectively then we define the propagated vertex-degree along $ \mu$ to equal $(a_1 + b_1-1, b_2, b_3, \cdots, b_m)$ if $k = 1$ and $(a_1, b_2, \cdots, b_{k-1}, a_1, b_{k+1}, \cdots, b_m)$ otherwise. In the latter case we will say that the propagated vertex degree along $ \mu$ is {\em mixed}.

\begin{proposition} \label{propagation} Assume that $\mu$ has a reversible end which gives rise to a directed path from $v$ to $w$ (i.e. $e_1$ is reversible). If $C^*(G, \chi)$ is exact then the propagated vertex degree along $\mu$ is not mixed.
\end{proposition}

\begin{proof}
Assume that the propagated vertex degree along $ \mu$ is mixed. For the proof we will relabel the edges so that the edges with range $v$ that share a color will have coloring $2$ and without loss of generality we will assume that $ k = 2$.  Notice that after appropriate reversing of edges we have $e_1, e_2, \cdots, e_{n-1}$ are reversible and $ s(e_n) = r(e_{n-1})$ it follows that $S_{\mu} := S_{e_n}S_{e_{n-1}} \cdots S_{e_1}$ is a partial isometry with $S_{\mu}^*S_{\mu} = P_v$ and $ S_{\mu}S_{\mu}^* = S_{e_n}S_{e_n}^* = P_w$.

We know that 
\begin{align*} S_{\mu}p_v C^*(G_2) p_v S_{\mu}^* &= S_{\mu} S_{\mu}^*S_{\mu}C^*(G_2)S_{\mu}^*S_{\mu} S_{\mu}^* \\ &= p_wS_{\mu} C^*(G_2) S_{\mu} p_w \\ & \in p_wC^*(G_2)p_w. 
\end{align*}
It follows that $p_w C^*(G, \chi)p_w$ contains a copy of $p_wC^*(G_1)p_w \free{\mathbb{C}} S_{\mu}p_vC^*(G_2)p_v S_{\mu}^*$ (here $\mathbb{C}$ is the subalgebra generated by $P_w$.

Now we have three cases:

Case 1: ($a_2 \geq 3$ or $b_1 \geq 3$.) 
Notice that by Proposition \ref{vertex-degree} $p_vC^*(G_2)p_v$ contains a subalgebra of the form $\mathbb{C}^a_1$ and $ p_wC^*(G_1)p_w$ contains a subalgebra of the form $\mathbb{C}^b_1$. It follows that $p_w C^*(G, \chi) p_w$ contains a copy of $\mathbb{C}^{a_2} \free{\mathbb{C}} \mathbb{C}^{b_2}$ and hence is not exact as in the third type of free product considered in Proposition \ref{non-exact}.

Case 2: (There is an edge $e$ with $s(e) = r(e) \in \{ v, w \}$ and another edge $f$ with $r(f) = r(e)$ and $ \chi(e) = \chi(f)$ .)
Just as in the proof of Proposition \ref{vertex-degree} we can see that $p_{r(e)}C^*(G_{\chi(e)})p_{r(e)}$ contains a copy of $ \mathbb{C}^3$.  Now mimicking the proof of the previous case we arrive at a non-exact $C^*$-subalgebra of $C^*(G, \chi)$.

Case 3: ($a_2 = 2$, $b_1 = 2$ and there are no loops based at $v$ and $w$.)
Just as in the proof of Proposition \ref{non-exact} we can see that $p_v C^*(G_2)p_v$ contains a copy of $M_2$ and $p_wC^*(G_1)p_w$ contains a copy of $M_2$ where $p_v$ and $p_w$ are the identities in the associated copies of $M_2$.  Then $S_{\mu}p_vC^*(G_2)p_v S_{\mu}^*$ contains a copy of $M_2$ for which $p_w$ is the identity and we are left with a subalgebra of $p_wC^*(G_2)p_w$ of the form $M_2 \free{\mathbb{C}} M_2$ which is not exact by Proposition \ref{non-exact}.
\end{proof}

Notice that if every possible propagation along paths is not mixed then we will say the graph is {\em propagation-exact}. 

We are now in a position to prove a theorem about edge-colored directed graph $C^*$-algebras in the case that $V_{\rm irr}$ contains at least two elements.

\begin{theorem} Let $(G, \chi)$ be an edge-colored directed graph in which $V_{\rm irr}$ consists of at least two vertices and is vertex-exact. then $C^*(G, \chi)$ is not exact if
\begin{itemize}
\item $(G, \chi)$ contains a reversible loop.
\item There is an edge $e$ in $E_v$ such that $M_{s(e)} \geq 2$.
\item The graph is not propagation-exact. 
\end{itemize}
otherwise $C^*(G, \chi)$ is nuclear.
\end{theorem}

\begin{proof}
The arguments in the proof of Theorem \ref{singletonexact} will apply to show that if $(G,\chi)$ contains a reversible loop or there is an edge $e \in E_v$ with $M_{s(e)} \geq 2$ then $C^*(G, \chi)$ is not exact.  Similarly if the graph is not propagation-exact then the preceding proposition tells us that $C^*(G, \chi)$ is not exact.

The only difference in the proof comes when we assume that $(G,\chi)$ does not contain a reversible loop, every edge in $M_{s(e)} \leq 1$ and the graph is propagation-exact.  In this case the only technicality may be if $u, v \in V_{\rm irr}$ and there is a path $\mu$ that propagates $u$ to $v$ and $\mu$ contains a cycle.  However since $(G, \chi)$ does not contain a reversible loop then the path $ \mu$ must connect $v$ to $w$ only through two edges $e$ and $f$ both in $E_v$.  Now, since $M_{s(e)} \leq 1$ and $ M_{s(e)} \leq 1$ we know that there is at most one cycle in $ \mu$.  Now if $h$ is an edge in $ \mu$ such that $\mu$ is not in the cycle then reverse the edge so that it points away from the cycle. After doing this to all edge in $ \mu$ that are not in the cycle then $\mu$ will consist of a $1$-colored subgraph of $(G, \chi)$. 

It follows that if $(G, \chi)$ does not contain a reversible loop, there is no edges in $E_v$ with $M_{s(e)} \geq 2$ and $(G, \chi)$ is propagation-exact then again there is a collection of reversing edges which results in $(G, \chi)$ turning into a $1$-colorable graph, and hence the $C^*$-algebra $C^*(G, \chi)$ is isomorphic to a nuclear $C^*$-algebra.
\end{proof}

In effect, when considering the algorithms we have the following corollary which characterizes completely the situation.

\begin{corollary} Let $(G, \chi)$ be an edge-colored directed then $C^*(G, \chi)$ is nuclear if and only if there is a collection of reversing reversible edges which results in a $1$-colored graph. If no such reversings exist then $C^*(G, \chi)$ is not exact.
\end{corollary}

Of course once we have that a graph is $1$-colorable we know that the universal edge-colored directed graph $C^*$-algebra is isomorphic to the reduced $C^*$-algebra of a separated graph.  We can now conclude the following. (Note that this gives a complete solution to Problem 7.2 of \cite{AraGoodearl}

\begin{theorem} $C^*(G, \chi)$ is nuclear if and only if $C^*_r(G, \chi) = C^*(G, \chi)$.
\end{theorem}

\begin{proof}
If $C^*(G, \chi)$ is nuclear then it is isomorphic to $C^*(H)$ where $H$ is a $1$-colored graph constructed by reversing reversible edges in $(G, \chi)$.  But by \cite[Theorem 3.8]{AraGoodearl} $C^*(H) = C^*_r(H)$ and then applying Proposition \ref{reduced} we have that $C^*_r(H) \cong C^*_r(G, \chi)$ and the forward direction follows.

On the other hand if $C^*(G, \chi) = C^*_r(G,\chi)$ then $C^*(G, \chi)$ is a nuclear $C^*$-algebra since the reduced $C^*$-algebras are always nuclear.
\end{proof}

\begin{acknowledgement} The author wishes to think two anonymous referees who found significant errors in previous versions, as well as providing suggestions which vastly improved both the results of the paper and the exposition throughout the paper. \end{acknowledgement}

\bibliographystyle{amsplain}

\end{document}